\documentclass[11pt]{article}

\usepackage[latin1]{inputenc}
\usepackage{t1enc}
\usepackage{amsmath, amssymb, amsfonts, amsthm, amsopn,epsfig}
\usepackage[none]{hyphenat}
\usepackage{tikz}
\usepackage{float}
\usepackage{graphicx}
\usepackage{caption}
\usepackage[margin=1in]{geometry}
\usepackage[toc,page]{appendix}
\usepackage{epstopdf}
\usepackage{etoolbox}
\usepackage[colorlinks=true]{hyperref}
\usepackage{dsfont}
\hypersetup{
	colorlinks=true,
	linkcolor=blue,
	filecolor=magenta,      
	urlcolor=cyan,
	citecolor=blue
}
\usepackage{cleveref}

\theoremstyle{plain}
\newtheorem{theorem}{Theorem}[section]

\newtheorem{corollary}[theorem]{Corollary}

\newtheorem{lemma}[theorem]{Lemma}
\newtheorem{conjecture}[theorem]{Conjecture}

\newcommand\mb[1]{\ensuremath{\mathbf{#1}}}

\theoremstyle{definition}

\title{Coloring lines and Delaunay graphs with respect to boxes}
\author{Istv\'an Tomon\thanks{Ume\r{a} University, \emph{e-mail}: \textbf{istvantomon@gmail.com}}}
\date{}

\begin{document}
	\sloppy 
	
	\maketitle

\begin{abstract}
    The goal of this paper is to show the existence (using probabilistic tools) of configurations of lines, boxes, and points with certain interesting combinatorial properties.

    (i) First, we construct a family of $n$ lines in $\mathbb{R}^3$ whose intersection graph is triangle-free of chromatic number $\Omega(n^{1/15})$. This improves the previously best known bound $\Omega(\log\log n)$ by Norin, and is also the first construction of a triangle-free intersection graph of simple geometric objects with polynomial chromatic number.


    (ii) Second, we construct a set of $n$ points in $\mathbb{R}^d$, whose Delaunay graph with respect to axis-parallel boxes has independence number at most $n\cdot (\log n)^{-(d-1)/2+o(1)}$. This extends the planar case considered by  Chen, Pach, Szegedy, and Tardos.
\end{abstract}

\section{Introduction}

The purpose of this paper is to present several interesting constructions of geometric configurations, with the help of simple probabilistic ideas.

\subsection{Coloring lines}

Given a graph $G$, how does its clique number $\omega(G)$ relate to its chromatic number $\chi(G)$? Clearly, $\chi(G)\geq \omega(G)$, but not much can be said in the other direction. Sophisticated probabilistic arguments show the existence of $n$-vertex triangle-free graphs of chromatic number $n^{1/2-o(1)}$, see e.g. \cite{K95}. The situation completely changes, however, if we assume that our graph is an intersection or disjointness graph of certain geometric shapes.

  The \emph{intersection graph} of a family $\mathcal{F}$ of sets is the graph, whose vertex set is $\mathcal{F}$, and two vertices are joined by an edge if they have a nonempty intersection. The \emph{disjointness graph} of a family is the complement of its intersection graph. In the past sixty years, the relationship between the clique number and chromatic number of intersection and disjointness graphs of certain geometric objects has been extensively studied. In almost all instances, the following tight connections have been observed. A family $\mathcal{G}$ of graphs is \emph{$\chi$-bounded}, if there exists a function $f$ such that $\chi(G)\leq f(\omega(G))$ for every member $G\in\mathcal{G}$. If $\mathcal{G}$ is a family of intersection or disjointness graphs of certain nice shapes (e.g. disks, boxes, segments, convex sets, curves), then either $\mathcal{G}$ is $\chi$-bounded \cite{AG60,CW21,D22,Gy85,K91,LMPT94,PTT,PT20,RW19}, or the chromatic number grows by at most a polylogarithmic function of the number of vertices assuming the clique number is fixed \cite{B65,FP14,KPW15,PaKK14,RW19,semilin21,boxes22,W15,W19}. In particular, a general result of Fox and Pach \cite{FP14} shows that if $G$ is the intersection graph of $n$ arc-wise connected sets in the plane (also known as a \emph{string graph}), then $\chi(G)\leq (\log n)^{O(\log \omega(G))}$. Another general result of Tomon \cite{semilin21} states that if the shapes can be defined by a \emph{semilinear relation}, then both the intersection and disjointness graph $G$ satisfies $\chi(G)\leq (\omega(G)\cdot \log n)^{O(1)}$, where the constant hidden in the $O(.)$ notation depends only on the complexity of the relation (we refer the interested reader to \cite{BCSTT20} or \cite{semilin21} for formal definitions).
  
  So far, the only family which does not fit into this pattern that we are aware of is the family of disjointness graphs of arcwise-connected sets in the plane. For every $n$, Suk and Tomon \cite{SukTom21} constructed such graphs on $n$ vertices that are triangle-free and have chromatic number $\Omega(n^{1/4})$. One goal of the current manuscript is to show another, perhaps even more natural family of geometric graphs with similar behavior.


  Pach, Tardos, and T\'oth \cite{PTT} asked whether the family of intersection graphs of \emph{lines} in $\mathbb{R}^3$ is $\chi$-bounded. This was answered by Norin \cite{Norin} in the \emph{negative} by showing that so called \emph{double shift graphs} can be realized as intersection graphs of lines. This infinite family of graphs was introduced by Erd\H{o}s and Hajnal \cite{EH} in 1964, and they proved that every $n$-vertex member of this family is triangle-free of chromatic number $\Theta(\log\log n)$. The argument of Norin was published by Davies \cite{D21}, who also proved that there are intersection graphs of lines of arbitrarily large girth and chromatic number. However, in Davies's construction the chromatic number grows even more slowly as a function of the number of vertices. Here, we show that somewhat surprisingly, there are triangle-free intersection graphs of lines with \emph{polynomial} chromatic number.

\begin{theorem}\label{thm:main}
There exists $c>0$ such that for every positive integer $n$, there exists a triangle-free intersection graph of $n$ lines in $\mathbb{R}^3$ of chromatic number at least $cn^{1/15}$.
\end{theorem}

In particular, we construct a triangle-free intersection graph of $n$ lines with independence number $O(n^{14/15})$. Previously, it was not known whether independence number $o(n)$ can be achieved for such graphs. Moreover, Theorem \ref{thm:main} suggests that the aforementioned result of Fox and Pach about string graphs is unlikely to have any reasonable extensions to higher dimensions. Finally, we remark that the family of disjointness graphs of lines in $\mathbb{R}^3$ is well-behaved,  Pach, Tardos, and T\'oth \cite{PTT} proved that it is $\chi$-bounded.


\subsection{Zarankiewicz problem for boxes}

The Zarankiewicz problem \cite{Z51} asks for the maximum number of edges in a bipartite graph on $n+n$ vertices containing no copy of the complete bipartite graph $K_{s,t}$. A geometric variant of this problem was recently introduced by Basit, Chernikov, Starchenko, Tao, and Tran \cite{BCSTT20}. Given a set $P$ and family of sets $\mathcal{F}$, the \emph{incidence graph} of $(P,\mathcal{F})$ is the bipartite graph with vertex classes $P$ and $\mathcal{F}$, where $p\in P$ and $F\in \mathcal{F}$ are joined by an edge if $p\in F$. 

Basit et al. \cite{BCSTT20} proved that if $G$ is the incidence graph of $n$ points and $n$ boxes (here and later, boxes are always axis-parallel) in $\mathbb{R}^d$, and $G$ is $K_{t,t}$-free, then $G$ has average degree $O_{d,t}((\log n)^{2d})$. This was recently improved to $O_d((\frac{\log n}{\log\log n})^{d-1})$ by Chan and Har-Peled \cite{CH23}, who  also highlighted that a construction matching their upper bound already appeared in a 1990 paper of Chazelle \cite{Ch90}.

\begin{theorem}[Chazelle \cite{Ch90}]\label{thm:box}
Let $d$ be a positive integer. Then there exists $c>0$ such that the following holds for every sufficiently large $n$. There exists a set of $n$ points $P$ and family of $n$ boxes $\mathcal{B}$ in $\mathbb{R}^d$ such that the incidence graph of $(P,\mathcal{B})$ is $K_{2,2}$-free of average degree at least $c(\frac{\log n}{\log\log n})^{d-1}$.
\end{theorem}

  We present this construction, as we build on it later. We highlight that the author of this paper \cite{semilin21} also proved that there exists a set of $n$ points and $n$ rectangles in the plane, whose incidence graph has girth $g$ and average degree $\Omega_g(\log\log n)$.

  \medskip

 Given a graph $G$, its \emph{separation dimension} is the smallest positive integer $d$ for which there exists an embedding $\phi:V(G)\rightarrow \mathbb{R}^d$ satisfying the following. If $\{x,y\}$ and $\{x',y'\}$ are disjoint edges of $G$, then the box spanned by $\phi(x)$ and $\phi(y)$ is disjoint from the box spanned by $\phi(x')$ and $\phi(y')$. Alon, Basavaraju, Chandran, Mathew, and Rajendraprasad \cite{ABCMR18} conjectured that for every $d$ there exists $c=c(d)$ such that every graph of separation dimension $d$ has average degree at most $c$. They proved the $d=2$ case of their conjecture, while Scott and Wood \cite{SW21} confirmed it in the case $d=3$. In general, Scott and Wood showed that an $n$-vertex graph of separation dimension $d$ has average degree at most $O_d((\log n)^{d-3})$. 
 
 However, it was observed by Tomon and Zakharov \cite{TZ21} that if $G$ is a $K_{2,2}$-free incidence graph of points and rectangles, then $G$ has separation dimension at most 4. Therefore, the $d=2$ case of Theorem \ref{thm:box} implies the existence of graphs on $n$ vertices of separation dimension 4, with average degree $\Omega(\frac{\log n}{\log\log n})$, thus disproving the conjecture. Following the ideas of \cite{TZ21}, we show that a $K_{2,2}$-free incidence graph of points and boxes in $\mathbb{R}^d$ has separation dimension at most $2d$, thus establishing the following corollary.

 \begin{corollary}\label{cor:sepdim}
 Let $d$ be a positive integer. For every sufficiently large $n$, there exists a graph on $n$ vertices with average degree at least $c(\frac{\log n}{\log\log n})^{d-1}$ of separation dimension at most $2d$, where $c=c(d)>0$ only depends on $d$.
 \end{corollary}
 
 Recently, the author of this paper \cite{semilin21}  studied Ramsey properties of so called semilinear graphs to model coloring properties of intersection and disjointness graph. A graph $G$ is \emph{semilinear} of complexity $t$, if the vertices of $G$ are points in some real space $\mathbb{R}^d$, and the edges depend only on the sign pattern of $t$ linear functions $f_1,\dots,f_t:\mathbb{R}^d\times \mathbb{R}^d\rightarrow \mathbb{R}$. In \cite{semilin21}, it is shown that every semilinear graph $G$  on $n$ vertices of complexity $t$ satisfies $\chi(G)\leq (\omega(G)\log n)^{O_t(1)}$.  We raised the question whether the power of $\log n$ in the upper bound  needs to increase with $t$. Combining Theorem \ref{thm:box} with ideas of \cite{semilin21}, one can show that there exist triangle-free semilinear graphs of complexity $O(d)$ on $n$ vertices with chromatic number $\Omega((\frac{\log n}{\log\log n})^{d-1})$, answering this question.

 \subsection{Delaunay graphs with respect to boxes}

 The \emph{Delaunay graph} of a set of points $P$ in the plane is the graph on vertex set $P$ in which $x$ and $y$ are joined by an edge if there exists a disk containing $x$ and $y$, and no other point of $P$. The Delaunay graph of every set of $n$ points is planar and thus contains an independent set of size at least $n/4$. As observed by Even, Lotker, Ron, and Smorodinsky \cite{ELRS}, this fact implies that any set of $n$ points has a \emph{conflict-free coloring} with respect to disks using $O(\log n)$ colors. That is, a coloring of $P$ such that for every disk $D$ containing a point of $P$ there is a color assigned to exactly one element of $P\cap D$. Conflict-free colorings are motivated by, for example, frequency assignment problems in cellular telephone networks. We refer the interested reader to the survey of Smorodinsky \cite{S13}.

 Generally, given a set of points $P$ in $\mathbb{R}^d$ and a collection $\mathcal{C}$ of subsets of $\mathbb{R}^d$, one can define the \emph{Delaunay graph of $P$ with respect to $\mathcal{C}$}, denoted by $D_{\mathcal{C}}(P)$, as the graph on vertex set $P$, in which $x$ and $y$ are joined by an edge if there exists some $C\in\mathcal{C}$ with $C\cap P=\{x,y\}$. Given a conflict-free coloring of $P$ with respect to $\mathcal{C}$ (defined analogously), every colorclass is an independent set of $D_{\mathcal{C}}(P)$. This motivates the study of independence numbers of Delaunay graphs.
 
 Even et al. \cite{ELRS} and Har-Peled, Smorodinsky \cite{HS} asked whether the Delaunay graph of a set of $n$ points in the plane with respect to rectangles contains an independent set of size $\Omega(n)$. This was disproved by Chen, Pach, Szegedy, and Tardos \cite{CPST}, who showed that this independence number might be as small as $O(\frac{n(\log\log n)^2}{\log n})$. On the other hand, the best known lower bound is $\Omega(n^{0.617})$ due to Ajwani, Elbassioni, Govindarajan, and Ray \cite{AEGR}. Here, we extend the upper bound for boxes in $\mathbb{R}^d$ for $d\geq 3$.

 \begin{theorem}\label{thm:Delaunay}
     For every $d\geq 3$ there exists $c>0$ such that the following holds for every sufficiently large $n$. There exists a set of $n$ points  in $\mathbb{R}^d$ whose Delaunay graph with respect to boxes has independence number at most $$\frac{cn(\log\log n)^{(d+3)/2}}{(\log n)^{(d-1)/2}}.$$
 \end{theorem}

 In particular, Chen, Pach, Szegedy, and Tardos \cite{CPST} proved that their upper bound is achieved by a uniform random point set in $[0,1]^{2}$ with high probability. Our upper bound in Theorem \ref{thm:Delaunay} is also achieved by a uniform random point set in $[0,1]^{d}$, after some modification. However, we note that our analysis of the Delaunay graph of a random point set is quite different from that of \cite{CPST}, which we were unable to extend already for the case $d=3$. We employ a more graph theoretic approach, involving the celebrated \emph{graph container method} \cite{KW82,S05}.

\bigskip
\noindent
 \textbf{Organization.} This paper is organized as follows. In Section \ref{sect:lines}, we prove Theorem \ref{thm:main}. Then, in Section \ref{sect:Zar}, we prove Theorem \ref{thm:box} and Corollary \ref{cor:sepdim}, and in Section \ref{sect:Del}, we prove Theorem \ref{thm:Delaunay}. We conclude the paper with some open problems and remarks. 
 
\section{Coloring lines --- Proof of Theorem \ref{thm:main}}\label{sect:lines}

In this section, we prove Theorem \ref{thm:main}. For integers $a\leq b$, let $[a,b]=\{a,\dots,b\}$ and $[a]=\{1,\dots,a\}$. First, we show that there is a set of $\Omega(N^{3/2})$ vectors in $[N]^3$ such that any three of them are linearly independent (over $\mathbb{R}$). This bound is the best possible \cite{BHPT}. Constructions of such sets and generalizations are already available \cite{BCV,BHPT,SudTom22}. We need some extra properties as well, and the most convenient way to ensure these properties is to give self-contained proof.   

Say that a vector $(a,b,c)\in \mathbb{Z}^3$ is \emph{indivisible} if the greatest common divisor of $a,b,c$ is 1.

\begin{lemma}\label{lemma:independent_vectors}
There exists $c,\varepsilon>0$ such that the following holds for every positive integer $N$. There exists a set $S\in [\lceil \varepsilon N\rceil, N]^3$ of size at least $c N^{3/2}$ such that any 3 elements of $S$ are linearly independent, and every element of $S$ is indivisible.
\end{lemma}

\begin{proof}
Let $p$ be a prime such that $\frac{1}{2}N^{3/2}<p<N^{3/2}$, which exists by Bertrand's postulate. First, we shall work over the field $\mathbb{F}_p$. For $\mb{u},\mb{v}\in\mathbb{F}_p^3\setminus\{0\}$, write $\mb{u}\sim \mb{v}$ if $\mb{u}=\lambda \mb{v}$ for some $\lambda\in\mathbb{F}_p$. Clearly, $\sim$ is an equivalence relation. A \emph{representative} of $\mb{u}$ is any element from its equivalence class.

Let $Q=[-N,N]^3$, then every nonzero element of $\mathbb{F}_p^3$ has a representative in $Q$, see e.g. Lemma 4.1 in \cite{SudTom22}. Furthermore, let $R$ be the set of vectors in $Q$ which have a coordinate in $[-\frac{N}{200},\frac{N}{200}]$, then $|R|\leq 3\cdot (2N)\cdot (2N)\cdot \frac{N}{100}<\frac{N^3}{8}$. Say that $\mb{v}\in \mathbb{F}_p^3$ is \emph{bad} if it has a representative in $R$, otherwise say that $\mb{v}$ is \emph{good}. Clearly, the number of bad elements is at most $|R|(p-1)< \frac{N^3}{8}(p-1)<\frac{p^3}{2}$. In particular, we get that at least half of $(\mathbb{F}_p\setminus\{0\})^3$ is good, we write $F$ for the set of good elements.

Now choose randomly three numbers $a,b,c\in \mathbb{F}_p\setminus \{0\}$ from the uniform distribution, independently from each other, and set $T=\{(a,bt,ct^2): t\in\mathbb{F}_p\setminus \{0\}\}$. The set $T$ is also known as a \emph{moment curve}. Firstly, we show that any three elements of $T$ are linearly independent. Otherwise, there exists $\mb{z}\in\mathbb{F}_p^3\setminus \{0\}$ and distinct $t_1,t_2,t_3\in \mathbb{F}_p$ such that $\langle \mb{z},(a,bt_i,ct_i^2)\rangle=0$ for $i\in [3]$. But this means that $t_1,t_2,t_3$ are distinct roots of the nonzero quadratic polynomial $f(x)=a\mb{z}(1)+b\mb{z}(2)\cdot x+ c\mb{z}(3)\cdot x^2$, a contradiction.
Secondly, note that for every $t\in \mathbb{F}_p\setminus \{0\}$, the vector $(a,bt,ct^2)$ is uniformly distributed in $(\mathbb{F}_p\setminus\{0\})^3$. Therefore, with probability at least $1/2$ the vector $(a,bt,ct^2)$ is good. This implies that there is a choice for $a,b,c$ such that at least half of the elements of $T$ are good. Fix such a choice. 

Now for each $\mb{u}\in F\cap T$, let $\mb{u}'\in \mathbb{Z}^3$ be a vector such that $\mb{u}'$ is a representative of $\mb{u}$ over $\mathbb{F}_p$, and the maximal absolute value of a coordinate of $\mb{u}'$ is minimal among such representatives. Clearly, $\mathbf{u}'$ is indivisible. Let $S_0=\{\mathbf{u}':\mathbf{u}\in F\cap T\}$. Any three vectors in $S_0$ are linearly independent over $\mathbb{Q}$, as they are linearly independent over $\mathbb{F}_p$. Also, for $\mb{x}\in S_0$ and $i\in [3]$, we have $\frac{N}{200}\leq |\mb{x}(i)|\leq N$. Finally, $|S_0|=|F\cap T|\geq \frac{p-1}{2}\geq \frac{N^{3/2}}{8}$. 

The set $S_0$ is almost what we want, we just need to get rid of the negative coordinates. Clearly, there exists $S_1\subset S_0$ of size at least $|S_1|\geq \frac{1}{8}|S_0|\geq \frac{N^{3/2}}{64}$ such that the sign-pattern of every vector in $S_1$ is the same. By flipping the negative coordinates of the elements of $S_1$, if necessary, we get a set $S$ with the desired properties. Hence, $c=\frac{1}{64}$ and $\varepsilon=\frac{1}{200}$ suffices.
\end{proof}

We remark that in order to get a bound of the form $n^{\Omega(1)}$ in Theorem \ref{thm:main}, it would have been enough to guarantee a set of size $N^{c}$ for any $c>0$ in Lemma \ref{lemma:independent_vectors}. It is not hard to argue that a random sample of $[\lceil \varepsilon N\rceil, N]^3$ of size $N^c$ gives a desired set if $c>0$ is sufficiently small.

\bigskip

Now we prepare the proof of Theorem \ref{thm:main}. In our arguments, we may assume that $n$ is sufficiently large, and we systematically omit the use of floors and ceilings whenever they are not crucial. Let $k>r$ be positive integers specified later with respect to $n$, and let $m:=\frac{k}{r}$. We assume that $k,r,m$ are also sufficiently large, which is ensured by their dependence on $n$.

Let $\varepsilon$ be the constant guaranteed by Lemma \ref{lemma:independent_vectors}, let  $N=\frac{m}{2}$, and fix a set $S\subset [\lceil \varepsilon N\rceil,N]^3$ satisfying the conditions of Lemma \ref{lemma:independent_vectors}. Then $|S|\geq c_s m^{3/2}$ for some absolute constant $c_s$, every coordinate of every element in $S$ is between $\frac{\varepsilon m}{2}$ and $\frac{m}{2}$, and any three vectors in $S$ are linearly independent. After removing some elements of $S$ arbitrarily, we may assume that $|S|=c_s m^{3/2}$.

For a vector $\mb{v}\in \mathbb{R}^3\setminus\{0\}$, say that a line $\ell$ in $\mathbb{R}^3$ is \emph{$\mb{v}$-type} if $\ell$ is parallel to $\mb{v}$. Let $\mathcal{L}$ be the set of lines $\ell$ such that $\ell$ is $\mb{v}$-type for some $\mb{v}\in S$, and $|\ell\cap [k]^{3}|\geq r$.

\begin{lemma}\label{lemma:size_of_L}
$|\mathcal{L}|\geq c_{\ell}\cdot \frac{k^3|S|}{r}$ for some constant $c_{\ell}>0$.
\end{lemma}

\begin{proof}
    Let $\mb{v}\in S$ and $\mb{x}\in [k/2]^3$. As every coordinate of $\mb{v}$ is at most $N=\frac{k}{2r}$, the line $$\ell_{\mb{x},\mb{v}}=\{\mb{x}+t\cdot \mb{v}:t\in\mathbb{R}\}$$ contains at least $r$ points of $[k]^3$, and so is contained in $\mathcal{L}$. On the other hand, as every coordinate of $\mb{v}$ is at least $\frac{\varepsilon m}{2}=\frac{\varepsilon k}{2r}$ and $\mb{v}$ is indivisible, $\ell_{\mb{x},\mb{v}}$ contains at most $\frac{r}{\varepsilon}$ points of $[k/2]^3$. This means that every $\ell\in \mathcal{L}$ coincides with at most $\frac{r}{\varepsilon}$ lines $\ell_{\mb{x},\mb{v}}$ with $(\mb{x},\mb{v})\in [k/2]^3\times S$. Hence, 
    $$|\mathcal{L}|\geq \frac{|S|\cdot (k/2)^3}{r/\varepsilon}=\frac{\varepsilon k^3|S|}{8r}$$
    which shows that $c_{\ell}:=\frac{\varepsilon}{8}$ suffices. 
\end{proof}

 Our strategy to find the desired configuration of lines is to take a random sample of $\mathcal{L}$, each element sampled independently with some appropriate probability $p$, and then clean this sample a bit to eliminate the triangles. The challenging part is then to show that the resulting graph has large chromatic number, which is ensured by estimating its independence number. In order to execute this strategy, we need to analyze certain properties of $\mathcal{L}$.

Let $G$ be the intersection graph of $\mathcal{L}$. Then the independence number $\alpha(G)$ satisfies $\alpha(G)\leq \frac{k^3}{r}$, as each line in $\mathcal{L}$ contains at least $r$ points of $[k]^3$. Furthermore, as any three elements of $S$ are linearly independent, we immediately get that if three lines of $\mathcal{L}$ have a  pairwise nonempty intersection, then they must go through the same point $\mb{z}\in \mathbb{R}^3$. As a precaution, it is worth pointing out that $\mathbf{z}$ need \emph{not} be an element of $[k]^3$.

For distinct $\mb{u},\mb{v}\in S$, let $\mathcal{H}_{\mb{u},\mb{v}}$ denote the set of planes $H$ such that $H$ contains both a $\mb{u}$-type and a $\mb{v}$-type line of $\mathcal{L}$. Also, let $$\mathcal{H}=\bigcup_{\substack{\mb{u},\mb{v}\in S\\ \mb{u}\neq \mb{v}}}\mathcal{H}_{\mb{u},\mb{v}}.$$
The following observation is crucial.

\begin{lemma}\label{lemma:size_of_H}
 $|\mathcal{H}_{\mb{u},\mb{v}}|\leq c_h\cdot\frac{k^{3}}{r^2}$ for some constant $c_h>0$.
\end{lemma}

\begin{proof}
We show that for every $H\in \mathcal{H}_{\mb{u},\mb{v}}$, we have $|H\cap [k]^3|\geq \frac{\varepsilon r^2}{16}$. Then the lemma follows by setting $c_h=\frac{16}{\varepsilon}$, as the elements of $\mathcal{H}_{\mb{u},\mb{v}}$  are pairwise disjoint.

Let $\ell\in \mathcal{L}$ be a $\mb{u}$-type line contained in $H$. Then there exists $\mb{x}\in [k]^3\cap \ell$ such that $$\mb{x}_i:=\mb{x}-i\cdot \mb{u}\in [k]^3\cap \ell$$
for $i=0,\dots,r-1$. Using that each coordinate of $\mb{u}$ is at least $\frac{\varepsilon m}{2}$ and that $\mb{x}_{r-1}\in [k]^3$, we deduce that every coordinate of  $\mb{x}_i$ is at least  $\frac{\varepsilon m}{2}\cdot (r-1-i)$. 

Let $\mb{y}_{i,j}=\mb{x}_i-j\cdot \mb{v}$ for $i,j=0,\dots,r$. As $H$ contains a $\mb{v}$-type line, we have $\mb{y}_{i,j}\in H$.  Furthermore, as every coordinate of $\mb{v}$ is at most $m$, we have for $b\in [3]$ that
$$\mb{y}_{i,j}(b)\geq \mb{x}_i(b)-j\cdot \mb{v}(b)\geq \frac{\varepsilon m}{2}\cdot (r-1-i)-j\cdot m.$$
Hence, for $i=0,\dots,\lfloor \frac{r}{2}\rfloor-1$ and $j=0,\dots,\lfloor \frac{\varepsilon r}{4}\rfloor-1$, every coordinate of $\mb{y}_{i,j}$ is positive (and from above bounded by $k$ trivially), so 
$\mb{y}_{i,j}$ is contained in $[k]^3$. This gives $$\left\lfloor \frac{r}{2}\right\rfloor \cdot \left\lfloor\frac{\varepsilon r}{4}\right\rfloor> \frac{\varepsilon r^2}{16}$$
distinct points of $[k]^3\cap H$,  finishing the proof.
\end{proof}

This lemma has a number of important consequences. First, we use it to count triangles.

\begin{lemma}\label{lemma:triangles}
The number of triangles in $G$ is at most $c_t\cdot \frac{|S|^3k^9}{r^6}$ for some constant $c_t>0$.
\end{lemma}

\begin{proof}
    Let $\ell_1,\ell_2,\ell_3$ be three lines forming a triangle in $G$, and let $\mb{v}_i\in S$ be the type of $\ell_i$ for $i\in [3]$. Clearly, $\mb{v}_1,\mb{v}_2,\mb{v}_3$ are pairwise distinct, and then linearly independent by the choice of $S$. Hence, as we remarked earlier, $\ell_1,\ell_2,\ell_3$ must go through the same point. Let $H_i$ be the unique plane containing $\ell_{i+1}$ and $\ell_{i+2}$ (indices meant modulo 3), then $H_i\in \mathcal{H}_{\mb{v}_{i+1},\mb{v}_{i+2}}$. The crucial observation is that the triple $(H_1,H_2,H_3)$ uniquely determines $(\ell_1,\ell_2,\ell_3)$ as $\ell_i=H_{i+1}\cap H_{i+2}$. 

   Therefore, the number of triangles of $G$ is upper bounded by the number of triples $(H_1,H_2,H_3)\in \mathcal{H}_{\mb{v}_2,\mb{v}_3}\times \mathcal{H}_{\mb{v}_3,\mb{v}_1}\times \mathcal{H}_{\mb{v}_1,\mb{v}_2}$, where $\mb{v}_1,\mb{v}_2,\mb{v}_3\in S$. The number of such triples is at most $|S|^3 \cdot (\frac{c_hk^3}{r^2})^3$ by Lemma \ref{lemma:size_of_H}. Hence, $c_t=c_h^3$ suffices.
\end{proof}

Let $\mathcal{J}$ denote the set of maximal independent sets of $G$ with respect to containment. Next, we estimate the size of $\mathcal{J}$. The author is grateful to J\'anos Pach and G\'abor Tardos for the elegant idea of the next proof, see also the concluding remarks of \cite{boxes22} where the same ideas are discussed.

\begin{lemma}\label{lemma:independent}
$|\mathcal{J}|\leq 2^{c_{j}|S|^2k^3/r^2}$ for some constant $c_{j}>0$.
\end{lemma}

\begin{proof}
    For $H\in\mathcal{H}$, let $B_H$ be the subgraph of $G$ induced by the lines of $\mathcal{L}$ contained in $H$. Then $B_H$ is a complete bipartite graph. Indeed, $H$ contains lines of exactly two types, and any two lines of different types in $H$ have a nonempty intersection. Furthermore, every edge of $G$ is contained in a unique $B_H$. Therefore, $\{B_H\}_{H\in \mathcal{H}}$ forms a partition of the edge set of $G$ into complete bipartite graphs.
    
    We write $X_H$ and $Y_H$ for the vertex classes of $B_H$. Furthermore, let $I$ be an independent set of $G$. Define the set $C(I)\subset V(G)$ as follows. As $I$ is an independent set, $I$ intersects at most one of $X_H$ and $Y_H$ for $H\in \mathcal{H}$, let $Z_H\in \{X_H,Y_H\}$ the one it intersects. If $I$ is disjoint from $X_H\cup Y_H$, then set $Z_H=X_H$. For $v\in V(G)$, put $v$ into $C(I)$ if and only if for every $H\in \mathcal{H}$ for which $v\in X_H\cup Y_H$ holds, we have $v\in Z_H$.

    First of all, note that $I\subset C(I)$ trivially. Furthermore, $C(I)$ is an independent set. Indeed, if $e$ is an edge of $G$, then it is covered by some $B_H$, and then the endpoint of $e$ not in $Z_H$ cannot be contained in $C(I)$. This implies that if $I$ is a maximal independent set, then $I=C(I)$. Hence, writing $\mathcal{C}=\{C(I):I\mbox{ is an independent set of }G\}$, we get $\mathcal{J}\subset \mathcal{C}$. On the other hand, each element of $\mathcal{C}$ is determined by the sequence $\{Z_H\}_{H\in\mathcal{H}}$, so 
    $$|\mathcal{C}|\leq 2^{|\mathcal{H}|}\leq 2^{c_h|S|^2k^3/r^2},$$
    where the last inequality holds by Lemma \ref{lemma:size_of_H}. Hence, $c_j:=c_h$ suffices.
\end{proof}

Now we randomly sample the elements of $\mathcal{L}$ independently with probability $p\in (0,1)$, and let $H$ be the subgraph of $G$ induced by the sampled elements. In what follows, we analyze the expected properties of $H$. During our arguments, we use standard concentration inequalities. We refer the reader to \cite{AS} as a general reference.

\begin{lemma}[Multiplicative Chernoff bound]\label{lemma:chernoff}
 Let $X$ be the sum of independent indicator random variables. If $\lambda\geq 2\mathbb{E}(X)$, then
 $\mathbb{P}(X\geq \lambda)\leq e^{-\lambda/6}$. Also, 
 $\mathbb{P}(X\leq \mathbb{E}(X)/2)\leq e^{-\mathbb{E}(X)/8}.$
\end{lemma}

Let $T$ be the number of triangles in $H$. The next lemma summarizes the expected properties of $H$. The constants $c_{\ell},c_t,c_j$ are the constants guaranteed by Lemmas \ref{lemma:size_of_L}, \ref{lemma:triangles}, \ref{lemma:independent}, respectively. 

\begin{lemma}\label{lemma:random}
Let $p\geq  \frac{6c_{j}|S|^2}{r}$. Then,
\begin{enumerate}
    \item  $\mathbb{P}(|V(H)|\geq \frac{c_{\ell} p k^3|S|}{2r})\geq \frac{3}{4}$,

    \item  $\mathbb{P}(T\leq \frac{4c_tp^3|S|^3k^9}{r^6})\geq \frac{3}{4}$,

    \item   $\mathbb{P}(\alpha(H)\leq \frac{2pk^3}{r})\geq \frac{3}{4}$.
\end{enumerate}
\end{lemma}

\begin{proof}
\begin{enumerate}
    \item  We have  $\mathbb{E}(|V(H)|)=p|\mathcal{L}|\geq \frac{c_{\ell} pk^3 |S|}{r}$. Hence, 
$$\mathbb{P}\left(|V(H)|< \frac{c_{\ell} p k^3|S|}{2r}\right)\leq \mathbb{P}\left(|V(H)|\leq \frac{1}{2}\mathbb{E}(|V(H)|)\right)\leq e^{-\mathbb{E}(|V(H)|)/8}<\frac{1}{4}.$$
Here, the second inequality holds by the multiplicative Chernoff bound, while the last inequality (generously) holds by our lower bound on  $\mathbb{E}(|V(H)|)$.

    \item  We have $\mathbb{E}(T)\leq \frac{c_tp^3|S|^3k^9}{r^6}$ by Lemma \ref{lemma:triangles}. Hence, by Markov's inequality, we can write
    $$\mathbb{P}\left(T>\frac{4c_tp^3|S|^3k^9}{r^6}\right)\leq \frac{\mathbb{E}(T)}{4c_tp^3|S|^3k^9/r^6}\leq \frac{1}{4}.$$

    \item Let $I\in \mathcal{J}$. As $\mathbb{E}(|I\cap V(H)|)=p|I|\leq \frac{pk^3}{r}$, we can apply the multiplicative Chernoff bound to get
$$\mathbb{P}\left(|I\cap V(H)|\geq \frac{2pk^3}{r}\right)\leq e^{-pk^3/3r}\leq 2^{-pk^3/3r}.$$
By the union bound and using Lemma \ref{lemma:independent},
$$\mathbb{P}\left(\exists I\in \mathcal{J}:|I\cap V(H)|\geq \frac{2pk^3}{r}\right)\leq |\mathcal{J}|\cdot 2^{-pk^3/3r}\leq 2^{c_{j}|S|^2k^3/r^2-pk^3/3r}.$$
Our lower bound on $p$ was chosen such that the right hand side is at most $2^{-c_{j}|S|^2k^3/r^2}<\frac{1}{4}$. Hence, with probability at least $3/4$, $H$ contains at most $\frac{2pk^3}{r}$ elements of every maximal independent set of $G$, which implies $\alpha(H)\leq \frac{2pk^3}{r}$.
\end{enumerate}
\end{proof}

From the previous estimates, we deduce the following.

\begin{lemma}\label{lemma:summary}
Let $p$ be such that 
\begin{equation}\label{equ:1}
\frac{6c_{j}|S|^2}{r}<p<\frac{c_{\ell}^{1/2}r^{5/2}}{4c_t^{1/2}|S|k^3}.
\end{equation}
Then there exists an induced subgraph $H'$ of $G$ such that $H'$ is triangle-free, $|V(H')|\geq \frac{c_{\ell} pk^3|S|}{4r}$ and $\alpha(H')\leq \frac{2pk^3}{r}.$
\end{lemma}

\begin{proof}
Let $H$ be defined as above. By Lemma \ref{lemma:random}, with probability at least $1/4$, $H$ satisfies the following conditions simultaneously: $|V(H)|\geq \frac{c_{\ell} p k^3|S|}{2r}$, $T\leq \frac{4c_tp^3|S|^3k^9}{r^6}$, and $\alpha(H)\leq \frac{2pk^3}{r}$. Fix some $H$ satisfying these conditions. Our upper bound on $p$ was chosen so that $|V(H)|\geq 2\cdot T$. Remove an arbitrary vertex from each triangle of $H$, and let $H'$ be the resulting graph. We have $|V(H')|\geq |V(H)|-T\geq \frac{1}{2}|V(H)|$, $H'$ is triangle-free, and $\alpha(H')\leq \alpha(H)$. Hence, $H'$ satisfies the desired conditions.
\end{proof}

Now we are ready to prove our main theorem. All that is left is to set the parameters $k,r$ appropriately and do a bit of calculation. 

\begin{proof}[Proof of Theorem \ref{thm:main}]
    In what comes, $c_0,c_1,c_2,c_3,c_4>0$ denote some unspecified constants, whose existence follows from simple calculations. Let $r:=c_0 k^{15/16}$, where $c_0$ is sufficiently large. Then $$|S|=c_s m^{3/2}=c_s \left(\frac{k}{r}\right)^{3/2}=\frac{c_s}{c_0^{3/2}}\cdot k^{3/32}.$$ Hence, the left-hand side of (\ref{equ:1}) is $$\frac{6c_{j}|S|^2}{r}=\frac{6c_j c_s^2 }{c_0^4}\cdot k^{-3/4},$$ while the right-hand side is
$$\frac{c_{\ell}^{1/2}r^{5/2}}{4c_t^{1/2}|S|k^3}=\frac{c_{\ell}^{1/2}c_0^{4}}{4c_t^{1/2} c_s}\cdot k^{-3/4}.$$ This shows that by choosing $c_0$ sufficiently large, we can ensure that the left-hand side is indeed smaller than the right-hand side. Also, we can choose $p=c_1 k^{-3/4}$ in between, that is, satisfying the inequalities in (\ref{equ:1}). By Lemma \ref{lemma:summary}, we get a triangle-free induced subgraph $H'$ of $G$ satisfying
$$|V(H')|\geq \frac{c_{\ell} pk^3|S|}{4r}=c_2 k^{45/32}$$ and 
$$\alpha(H')\leq \frac{2pk^3}{r}=c_3 k^{42/32}.$$ 
Finally, choose $k$ such that $n=c_2 k^{45/32}$ holds. Then the graph $H'$ satisfies $|V(H')|\geq n$ and $\alpha(H')\leq c_4 n^{14/15}$. We can remove further vertices of $H'$ arbitrarily to get a graph $H''$ with exactly $n$ vertices. The graph $H''$ is a triangle-free intersection graph of $n$ lines in $\mathbb{R}^3$ and  $$\chi(H'')\geq \frac{n}{\alpha(H'')}\geq  \frac{1}{c_4}n^{1/15},$$ finishing the proof. 
\end{proof}

\section{Zarankiewicz for Boxes --- Proof of Theorem \ref{thm:box}}\label{sect:Zar}

In this section, we prove Theorem \ref{thm:box} and Corollary \ref{cor:sepdim}. Unlike in the previous section, given real numbers $a,b$, $[a,b]$ denotes the closed real interval with endpoints $a$ and $b$. We fix some positive integer parameters first. Let $k$ be the solution of the equation  
$$n=(100k^{2d})^{k(d-1)}\cdot \binom{k+d-1}{d-1}$$ (where we omit the detail that $k$ might not be an integer). Then $k=\Theta_d(\frac{\log n}{\log\log n})$. Furthermore, define $s:=100k^{2d}$ and $m:=s^k$, and observe that $n=m^{d-1}\binom{k+d-1}{d-1}$. Also, $n$ being sufficiently large with respect to $d$ ensures that $s,k,m$ are also sufficiently large.

\medskip

First, we define our family of boxes $\mathcal{B}$. Given $\textbf{t}\in \mathbb{N}^d$ and $\textbf{p}\in \mathbb{Z}^d$, let $B^{(s)}_{\textbf{t}}(\textbf{p})=B_{\textbf{t}}(\textbf{p})$  denote the $s^{\textbf{t}(1)}\times\dots\times s^{\textbf{t}(d)}$ sized box  $$\prod_{i=1}^{d}\left[\ s^{\textbf{t}(i)}\mathbf{p}(i),\ s^{\textbf{t}(i)}\mathbf{p}(i)+s^{\textbf{t}(i)}\ \right).$$ Call $B_{\textbf{t}}(\textbf{p})$ a \emph{$\textbf{t}$-block}, or simply a \emph{block}. Clearly, the \emph{$\textbf{t}$-blocks} partition $\mathbb{R}^d$ for every $\textbf{t}\in \mathbb{N}^d$. Furthermore, for $\ell\in\mathbb{N}$, let $$T_{\ell}=\left\{\textbf{t}\in \mathbb{N}^d: \sum_{i=1}^d \mathbf{t}(i)=\ell\right\},$$
then $|T_{\ell}|=\binom{\ell+d-1}{d-1}$. Finally, let $\mathcal{B}$ be the family of all blocks of volume $m=s^{k}$ contained in $[0,m]^d$. Formally, $\mathcal{B}$ is the family of all blocks $B_{\textbf{t}}(\textbf{p})$, where $\textbf{t}\in T_k$ and $\textbf{p}(i)\in \{0,\dots,s^{k-\mathbf{t}(i)}-1\}$ for every $i\in [d]$. Note that  $|\mathcal{B}|=|T_k|\cdot m^{d-1}=n$ by the choice of our parameters. We remark that this family of boxes is also studied in a recent work of the author \cite{boxes22} in the context of piercing numbers. An important property of $\mathcal{B}$ is that for any distinct $B,B'\in \mathcal{B}$, the intersection $B\cap B'$ is either empty, or it is a $\mathbf{t}$-block for some $\mathbf{t}$ satisfying $\sum_{i=1}^{d}\mathbf{t}(i)\leq k-1$. To this end, define
$$\mathcal{B}^-=\{B\subset [0,m]^d:B\mbox{ is a }\mathbf{t}\mbox{-block for some }\mathbf{t}\in T_{k-1}\}.$$ 
Then for every distinct $B,B'\in\mathcal{B}$, $B\cap B'$ is empty or contained in an element of $\mathcal{B}^-$. Note that the volume of every element of $\mathcal{B}^-$ is $\frac{m}{s}$ and $|\mathcal{B}^-|=|T_{k-1}|sm^{d-1}$.

 \medskip

Now, we define the set of points $P$. In case $d=2$, there is a perfect set for our purposes (well known in discrepancy theory), the van der Corput set \cite{vdc}. Assuming $n=2^t$, and writing numbers in binary representation, the points of this set are $(0.x_1\dots x_t,0.x_t\dots x_1)$, where $(x_1,\dots,x_t)\in \{0,1\}^t$. The van der Corput set has the property that every rectangle of area $\frac{1}{n}$ contains at most one of its points. Scaling the unit square by a factor of $m$, we get a set of $n$ points $P$ in $[0,m]^2$ such that no rectangle of area $\frac{m^2}{n}$ contains two points of $P$. Here, by our choice of parameters, we have $\frac{m^2}{n}=\frac{m}{|T_k|}=\frac{m}{k+1}>\frac{m}{s}$. Let $G$ be the incidence graph of $(P,\mathcal{B})$. Then every point in $P$ is contained in exactly $|T_k|=k+1$ rectangles of $\mathcal{B}$, so $G$ has average degree $k+1=\Omega(\frac{\log n}{\log\log n})$. Also, the incidence graph of $G$ is $K_{2,2}$-free, as the intersection of any two rectangles in $\mathcal{B}$ has area at most $\frac{m}{s}$. This concludes the case $d=2$.

\medskip

Now we assume that $d\geq 3$. We show that instead of a well structured set such as the previously described van der Corput set, a random set of points also works (which works in the case $d=2$ as well).

\begin{lemma}\label{lemma:pointset}
There exists a set of $n$ points $P$ in $[0,m]^d$ such that every block $B\in\mathcal{B}^-$ contains at most one element of $P$. Moreover, the number of triples of points of $P$ that are contained in the same block $B\in \mathcal{B}$ is at most $32 |T_k|^3 n$.
\end{lemma}

\begin{proof}
First, let $Q$ be a set of $2n$ points chosen randomly and independently in $[0,m]^d$ from the uniform distribution. We show that after some cleaning, i.e. deleting some of the points, we get a $K_{2,2}$-free incidence graph with large average degree.

Say that a pair of distinct points $x,y\in Q$ is \emph{bad} if there exists a block $B\in\mathcal{B}^{-}$ with $x,y\in B$. Let $X$ be the number of bad pairs. For any pair $x,y$ of independent points from the uniform distribution on $[0,m]^d$, and $B\in\mathcal{B}^{-}$, we have $\mathbb{P}(x,y\in B)=\frac{1}{(sm^{d-1})^2}$. Hence, $$\mathbb{P}(\exists B\in\mathcal{B}^{-}:x,y\in B)\leq \frac{|\mathcal{B}^{-}|}{s^2m^{2(d-1)}}=\frac{|T_{k-1}|}{sm^{d-1}}.$$ But then,
 $$\mathbb{E}(X)\leq (2n)^2\cdot \frac{|T_{k-1}|}{sm^{d-1}}=\frac{4n|T_k|\cdot|T_{k-1}|}{s}< \frac{4nk^{2d}}{s}\leq \frac{n}{4},$$
By Markov's inequality, $\mathbb{P}(X\geq n)\leq \frac{1}{4}$.

 Also, let $T$ be the number of triples of points of $Q$ that are contained in the same block $B\in\mathcal{B}$. For any triple $x,y,z$ of independent points from the uniform distribution on $[0,m]^d$, we have $\mathbb{P}(x,y,z\in B)=\frac{1}{m^{3(d-1)}}$. Hence, $\mathbb{P}(\exists B\in\mathcal{B}:x,y,z\in B)\leq \frac{|\mathcal{B}|}{m^{3(d-1)}}=\frac{|T_k|}{m^{2(d-1)}}$. But then,
 $$\mathbb{E}(T)\leq (2n)^3\cdot \frac{|T_k|}{m^{2(d-1)}}=8|T_k|^3 n.$$
 By Markov's inequality, $\mathbb{P}(T\geq 32 |T_k|^3 n)\leq \frac{1}{4}$.
 
 Hence, there exists a set of points $Q$ satisfying $X\leq n$ and $T\leq 32 |T_k|^3 n$. Fix such a set. Let $P'$ be the set we get by removing all points of $Q$ that are contained in a bad pair. Then we removed at most $X$ points in total, so $|P'|\geq n$. Remove some further points of $P'$ arbitrarily to get a set $P$ with exactly $n$ elements. This concludes the construction of our point set.
 \end{proof}
 
We remark  that the condition on the number of triples in not needed for the proof of Theorem~\ref{thm:box}, but it  be useful later.
 
 Let $P$ be a set guaranteed by Lemma \ref{lemma:pointset}, and we analyze the incidence graph $G$ of $(P,\mathcal{B})$. Note that the degree of every element in $P$ is exactly 
 $$|T_k|=\binom{k+d-1}{d-1}\geq \frac{k^{d-1}}{(d-1)!}\geq c\left(\frac{\log n}{\log \log n}\right)^{d-1}$$
 where $c>0$ is some appropriate constant depending only on $d$. Hence, $G$ has average degree $|T_k|\geq c(\frac{\log n}{\log\log n})^{d-1}$.  Finally, $G$ is $K_{2,2}$-free, as no $\mathbf{t}$-block for $\mathbf{t}\in T_{k-1}$ contains more than one element of $P$. This finishes the proof of Theorem \ref{thm:box}.

\bigskip

We conclude this section with the proof of Corollary \ref{cor:sepdim}.

\begin{proof}[Proof of Corollary \ref{cor:sepdim}]
Let $P$ be a set of $n$ points and $\mathcal{B}$ be a set of $n$ boxes in $\mathbb{R}^d$ guaranteed by Theorem \ref{thm:box}.  Let $G$ be the incidence graph of $(P,\mathcal{B})$, and recall that $G$ is $K_{2,2}$-free. We show that $G$ has separation dimension at most $2d$, which then finishes the proof. Our task is to find an embedding $\phi:V(G)\rightarrow \mathbb{R}^{2d}$ such that if $p,p'\in P$ and $B,B'\in \mathcal{B}$ such that $p\in B$ and $p'\in B'$, then the box spanned by $\phi(p)$ and $\phi(B)$ is disjoint from the box spanned by $\phi(p')$ and $\phi(B')$. 

We define $\phi:V(G)\rightarrow \mathbb{R}^{2d}$ as follows. If $p\in P$, then $\phi(p)(i)=p(i)$ and $\phi(p)(i+d)=-p(i)$ for $i\in [d]$. Also, given a box $B\in \mathcal{B}$ such that $B=[a_1,b_1]\times \dots\times [a_d,b_d]$, let $\phi(B)(i)=b_i$ and $\phi(B)(i+d)=-a_i$ for $i\in [d]$. We show that $\phi$ suffices.

Define the partial ordering $\prec$ on $\mathbb{R}^{2d}$ such that $x\preceq y$ if $x(i)\leq y(i)$ for every $i\in [d]$. Observe that $p\in B$ for some $p\in P$ and $B\in\mathcal{B}$ if and only if $\phi(p)\preceq \phi(B)$. Also, the box spanned by $p$ and $B$ is exactly the set of points $y\in\mathbb{R}^{2d}$ such that $\phi(p)\preceq y\preceq \phi(B)$. Therefore, suppose that $p,p'\in P$ and $b,b'\in B$ are such that $p\in B$, $p'\in B'$, and the box spanned by $\phi(p)$ and $\phi(B)$ intersects the box spanned by $\phi(p')$ and $\phi(B')$ in some $y\in \mathbb{R}^{2d}$. Then $\phi(p),\phi(p')\prec y\prec \phi(B),\phi(B')$, which implies that $p'\in B$ and $p\in B'$ as well. But then $\{p,p',B,B'\}$ are the vertices of a copy of $K_{2,2}$ in $G$, contradiction.
\end{proof}

\section{Delaunay graphs --- Proof of Theorem \ref{thm:Delaunay}}\label{sect:Del}

In this section, we prove Theorem \ref{thm:Delaunay}. The Delaunay graph of a point set always refers to its Delaunay graph with respect to boxes. We give a brief outline of our strategy. 

\bigskip
\noindent
\textbf{Proof strategy.} Let $\mathcal{B}$  and $\mathcal{B}^-$ be the same systems of blocks as defined in the previous section (with the parameters $k,s,m$ slightly changed with respect to $n$). Given a set of points $P\subset [0,m]^d$, one can define the graph $G=G_P$ by connecting two points with an edge if they are contained in the same block $B\in \mathcal{B}$. Observe that if $B$ contains exactly two points $x,y\in P$, then the edge connecting $x$ and $y$ is also an edge of the Delaunay graph of $P$. Our aim is to find a set of $n$ points $P$ such that every block $B\in \mathcal{B}$ contains at most two points of $P$ and $G_P$ has small independence number. Then $G_P$ is a spanning subgraph of the Delaunay graph of $P$, which then also has small independence number.

In order to find $P$, we first consider a larger set $Q$ with the property that every $B\in\mathcal{B}^-$ contains at most one point of $Q$. Such a set is already constructed in the previous section. Our goal is to show that a random sample $P_0$ of $Q$ with some appropriate probability is close to our desired set $P$. In order to control the independence number of $G_{P_0}$, we  apply the graph container method. That is, we show that there is a small collection of small subsets of $Q$ such that every independent set of $G_Q$ is contained in one of these sets. The existence of such a collection follows from a supersaturation result. More precisely, the result we need and prove is that every subset of $Q$ of size $\lambda m^{d-1}$ ($\lambda\geq 2$) induces a subgraph in $G_Q$ of maximum degree at least $\frac{\lambda}{2} |T_k|$. The proof of this property uses the fact that every $B\in\mathcal{B}^-$ contains at most one point of $Q$.

\bigskip

We  now execute the above strategy formally. Let $N$ be a parameter specified later, and define $k,s,m$ as in the previous section, but now with respect to $N$ instead of $n$. That is $N=(100k^{2d})^{k(d-1)}\binom{k+d-1}{d-1}$, $s=16k^{2d}$ and $m=s^k$, and these parameters satisfy $k=\Theta_d((\frac{\log N}{\log\log N})^{d-1})$, $N=m^{d-1}\binom{k+d-1}{d-1}$. Also, define blocks, $\mathcal{B}$ and $\mathcal{B}^-$ in the same manner. Then by Lemma \ref{lemma:pointset}, there exists a set $Q\subset [0,m]^d$ of $N$ points such that no block in $\mathcal{B}^-$ contains more than one element of $Q$, and the number $T$ of triples of points contained in some block of $\mathcal{B}$ is at most $32 |T_k|^3 N$.

Define the graph $G$ on vertex set $Q$ such that two points are connected by an edge if they are contained in the same block $B\in\mathcal{B}$. Note that each edge of $G$ comes from a unique block $B$. Next, we  randomly sparsify the set $Q$ to get a set $P_0$ which is close to our desired set $P$. In order to control the independence number of a random induced subgraph of $G$, we  employ the celebrated graph container method \cite{KW82,S05}. We would like to show that there is small collection $\mathcal{C}$ of small subsets of $Q$, called \emph{containers}, such that every independent set of $G$ is contained in some element of $\mathcal{C}$. In order to show the existence of such a collection, one needs to ensure that large subsets of $Q$ induce subgraphs of $G$ of large maximum degree. We prove such a result in the next lemma. 

\begin{lemma}\label{lemma:saturation}
Let $\lambda\geq 2$ and $C\subset Q$ such that $|C|\geq \lambda \cdot m^{d-1}$. Then the maximum degree of $G[C]$ is at least $\frac{\lambda}{2} |T_k|$.
\end{lemma}

\begin{proof}
Let $B\in \mathcal{B}$ and let $n_B=|B\cap C|$. As $G[B\cap C]$ is a complete graph,  $$e(G[B\cap C])=\frac{n_B^2-n_B}{2}.$$ Therefore,
$$e(G[C])=\sum_{B\in\mathcal{B}}e(G[B\cap C])=\sum_{B\in\mathcal{B}}\left(\frac{n_B^2}{2}-\frac{n_B}{2}\right)\geq \frac{(\sum_{B\in\mathcal{B}} n_B)^2}{2|\mathcal{B}|}-\frac{\sum_{B\in\mathcal{B}} n_B}{2}.$$
Here, the first equality holds by the fact that every edge of $G$ is contained in exactly one of the blocks, and the last inequality is due to the Cauchy-Schwartz inequality. But note that $\sum_{B\in\mathcal{B}} n_B=|T_k|\cdot |C|$, as each element of $Q$ is contained in exactly $|T_k|$ blocks of $\mathcal{B}$. Therefore, we can rewrite the right hand side as
$$e(G[C])\geq  \frac{|T_k|^2 |C|^2}{2|T_k|\cdot m^{d-1}}-\frac{|T_k| |C|}{2}=\frac{\lambda-1}{2}\cdot |T_k|\cdot |C|\geq \frac{\lambda}{4}\cdot |T_k|\cdot |C|.$$
This shows that the average degree of $G[C]$ is at least $\frac{\lambda}{2} |T_k|$, finishing the proof.
\end{proof}

Now we are ready to state and prove our container lemma. The proof should be mostly standard to anyone familiar with the container method.

\begin{lemma}\label{lemma:container}
There exists a collection $\mathcal{C}\subset 2^Q$ with the following properties.
\begin{description}
    \item[(i)] Every independent set of $G$ is contained in some $C\in\mathcal{C}$.
    \item[(ii)] $|C|\leq 3m^{d-1}$ for every $C\in\mathcal{C}$.
    \item[(iii)] $|\mathcal{C}|\leq \exp\left(\frac{cm^{d-1}}{|T_k|}(\log\log N)^2\right)$ for some $c_1=c_1(d)$.
\end{description}
\end{lemma}

\begin{proof}
Let $<$ be an arbitrary total ordering of the elements of $Q$. For a subgraph $H$ of $G$ and vertex $v\in V(H)$, $N_{H}(v)=\{w\in V(H):vw\in E(H)\}$ denotes the \emph{neighborhood} of $v$ in $H$.

Fix an independent set $I$ of $G$. We construct a \emph{fingerprint} $S$ and a set $f(S)$ (depending only on $S$) for $I$ with the help of the following algorithm.

Let $S_0=\emptyset$ and $G_0=G$. If $S_i$ and $G_i$ are already defined, we define $S_{i+1}$ and $G_{i+1}$ in the following manner. Let $v\in R$ be the first vertex (with respect to $<$) of maximum degree in $G_{i}$.  
\begin{itemize}
\item If $|V(G_{i})|\leq  2m^{d-1}$, then stop, and set $S:=S_{i}$ and $f(S):=V(G_{i})$.
\item  Otherwise, if $v\not\in I$, then set $S_{i+1}=S_i$, remove $v$ from $G_i$, and let the resulting graph be $G_{i+1}$. 
\item If $v\in I$, then set $S_{i+1}=S_i\cup\{v\}$, and remove $\{v\}\cup N_{G_i}(v)$ from $G_i$, let $G_{i+1}$ be the resulting graph.
\end{itemize}

We analyze this algorithm. At each step, the size of $G_i$ decreases, so the algorithm stops after a finite number of steps. The first observation one has to make is that $f(S)$ indeed only depends on $S$. We omit the details, as this argument is standard (one has to check that at every step, $G_i$ only depends on $S_i$). Secondly, as $I$ is an independent set, we have $I\subset S_i\cup V(G_i)$ for every $i$, so in particular $I\subset S\cup f(S)$. 

Finally, we have $|S|\leq \frac{c_0 m^{d-1}}{|T_k|}\log\log N$ for some $c_0=c_0(d)$. Indeed, in case we added a vertex $v$ to $S_i$ to get $S_{i+1}$, we removed at least $1+|N_{G_i}(v)|$ vertices from $G_i$ to get $G_{i+1}$. But $v$ is a vertex of maximum degree in $G_i$, and $|V(G_i)|\geq 2m^{d-1}$,  so we can apply Lemma \ref{lemma:saturation} with $\lambda=\frac{|V(G_i)|}{m^{d-1}}$ to get $$1+|N_{G_i}(v)|> \frac{|V(G_i)| |T_k|}{2m^{d-1}}.$$ Therefore, $$|V(G_{i+1})|\leq \left(1-\frac{|T_k|}{2m^{d-1}}\right)|V(G_i)|<\exp\left(-\frac{|T_k|}{2m^{d-1}}\right)|V(G_i)|.$$
From this, we deduce that $\exp\left(-\frac{|S|\cdot |T_k|}{2m^{d-1}}\right)\cdot |Q|\geq 2m^{d-1}$. Solving the inequality gives $$|S|\leq \frac{2m^{d-1}}{|T_k|}\cdot \log \frac{|Q|}{2 m^{d-1}}\leq \frac{c_0 m^{d-1}}{|T_k|}\log\log N,$$
where $c_0$ is some constant depending only on $d$. In particular, we have $|S\cup f(S)|\leq 3m^{d-1}$.

Let $\mathcal{C}$ be the collection of all the sets $S\cup f(S)$, where $S$ is the fingerprint of some independent set $I$. Then (i) and (ii) are satisfied. Also, as each fingerprint has size at most $z:=\frac{c_0 m^{d-1}}{|T_k|}\log\log N$, we can bound the size of $\mathcal{C}$ by simply counting all subsets of $V(G)$ of size at most $z$. Therefore, 
$$|\mathcal{C}|\leq \sum_{i=0}^{z}\binom{|Q|}{i}\leq \left(\frac{4N}{z}\right)^{z}\leq \exp\left(\frac{c_1m^{d-1}}{|T_k|}(\log\log N)^2\right)$$
with some appropriate $c_1=c_1(d)$, where the last inequality holds by observing that $\frac{N}{z}$ is polylogarithmic in $N$. 
\end{proof}

Let $n=\frac{N}{400|T_k|^{3/2}}=\frac{m^{d-1}}{400|T_k|^{1/2}}$ (or rather, fix the parameter $k$ with respect to $n$ such that this equality is satisfied), and let $p:=\frac{4n}{N}=\frac{1}{100|T_k|^{3/2}}$. We highlight that $\log n=(1+o(1))\log N$, which is used subtly in calculations later. Let $P_0$ be the random sample  we get by selecting each element of $Q$ independently with probability $p$. Then $\mathbb{E}(|P_0|)=4n$, so by the multiplicative Chernoff bound, we have $\mathbb{P}(|P_0|\geq 2n)\geq 3/4$.

Let $T'$ be the number of triples in $P_0$ that are contained in the same block $B\in\mathcal{B}$. Then $$\mathbb{E}(T')=p^3 T \leq 32\left(\frac{3n}{N}\right)^3 |T_k|^3 N \leq \frac{n}{4}.$$ Hence, by Markov's inequality, we have $\mathbb{P}(T\geq n)\leq \frac{1}{4}$.

Finally, let $H_0=G[P_0]$ and $\alpha=\frac{10c_1m^{d-1}}{|T_k|}(\log\log N)^2$, where $c$ is the constant given by Lemma~\ref{lemma:container}.

\begin{lemma}
$\alpha(H_0)\leq \alpha$ with probability at least $\frac{3}{4}$.
\end{lemma}

\begin{proof}
Let $\mathcal{C}$ be the collection given by Lemma \ref{lemma:container}. Then every $C\in\mathcal{C}$ satisfies $|C|\leq 3m^{d-1}$. In particular, $$\mathbb{E}(|C\cap P_0|)\leq 3p m^{d-1}<\frac{\alpha}{2}.$$ But then by the multiplicative Chernoff bound, 
$$\mathbb{P}(|C\cap P_0|\geq \alpha)\leq e^{-\alpha/6}.$$
Furthermore, by the union bound,
$$\mathbb{P}(\exists C\in\mathcal{C}:|C\cap P_0|\geq \alpha)\leq |\mathcal{C}|e^{-\alpha/6}<\frac{1}{4}.$$
Therefore, with probability at least $\frac{3}{4}$, we have $|C\cap P_0|\leq \alpha$ for every $C\in\mathcal{C}$. But every independent set of $G$ is contained in some element of $\mathcal{C}$, so $\alpha(H_0)\leq \alpha$ with probability at least~$\frac{3}{4}$.
\end{proof}

Hence, with positive probability, there exists $P_0$ such that $|P_0|\geq 2n$, $T'\leq n$ and $\alpha(H_0)\leq \alpha$. Fix such a set $P_0$. Let $P_1$ be the set of points we get by removing a member of every triple of $P_0$ that is contained a block of $\mathcal{B}$. Then $|P_1|\geq |P_0|-T'\geq n$, so we can remove some further points arbitrarily to get a set $P\subset P_1$ with exactly $n$ elements. Define $H=H_0[P]=G[P]$. Then 
$$\alpha(H)\leq \alpha(H_0)\leq \alpha=\frac{10c_1m^{d-1}}{|T_k|}(\log\log N)^2=\frac{4000c_1 n}{|T_k|^{1/2}}(\log\log N)^2\leq \frac{c n (\log \log n)^{(d+3)/2}}{(\log n)^{(d-1)/2}},$$ for some constant $c$ depending only on $d$.  Furthermore, every block of $B\in\mathcal{B}$ contains at most two points of $P$. The latter ensures that $H$ is a spanning subgraph of the Delaunay graph $D$ of $P$, so $\alpha(D)\leq \alpha(H)$. This finishes the proof of Theorem \ref{thm:Delaunay}.

\section{Concluding remarks}

\subsection{Coloring Lines}

As mentioned in the introduction, Davies \cite{D21} proved that for every pair of positive integers $g$ and $\chi$ there is an intersection graph of lines with girth $g$ and chromatic number $\chi$. We believe that for any fixed girth $g$, the chromatic number grows polynomially as a function of the number of lines.

\begin{conjecture}
    For every $g\in\mathbb{N}$ there exists $\varepsilon>0$ such that the following holds. For every sufficiently large $n$, there exists an intersection graph of $n$ lines in $\mathbb{R}^3$ of girth at least $g$ and chromatic number at least $n^{\varepsilon}$.
\end{conjecture}

Furthermore, it would be interesting to see whether Theorem \ref{thm:main} can be extended to the projective space $\mathbb{PR}^3$. The construction of Norin (see \cite{D21}) shows that there are triangle-free intersection graphs of lines in $\mathbb{PR}^3$ of arbitrarily large chromatic number. However, in the proof of Theorem \ref{thm:main}, and in the aforementioned construction of Davies as well, it is crucial to use large sets of parallel lines.

\begin{conjecture}
    There exists $\varepsilon>0$ such that for every sufficiently large $n$, there exists a triangle-free intersection graph of $n$ lines in $\mathbb{PR}^3$ of chromatic number at least $n^{\varepsilon}$.
\end{conjecture}

We remark that if we work in the finite projective space $\mathbb{PF}_p^3$, then there is a family of $n=n(p,g)$ lines, whose intersection graph has girth more than $g$ and chromatic number at least $n^{1/g-o(1)}$ (unpublished). This shows that if the previous two conjectures fail, then the reason must be geometric rather than algebraic.

\subsection{Small Independent sets}

 In order to show that a family $\mathcal{G}$ of graphs  is not $\chi$-bounded, it is enough to find members $G\in \mathcal{G}$ that are triangle-free of independence number $o(v(G))$. We proved Theorem \ref{thm:main} by constructing such intersection graphs of lines. Suk and Tomon \cite{SukTom21} showed that the family of disjointness graphs of curves, and  Walczak \cite{W15} showed that the family of intersection graphs of segments in the plane contain triangle-free $n$-vertex graphs of independence number $o(n)$. However, it remains open whether a similar statement holds for the intersection graph of boxes in 3 or higher dimensions, see also \cite{W15} for the same question raised.

\begin{conjecture}
For every $\alpha>0$ there exist $n$ and a triangle-free intersection graph of $n$ boxes in $\mathbb{R}^3$ with independence number at most $\alpha n$. 
\end{conjecture}

\subsection{Delaunay graphs and Posets}

Define the partial ordering $\prec$ on $\mathbb{R}^d$ by writing $\textbf{x}\preceq \textbf{y}$ if $\textbf{x}(i)\leq \textbf{y}(i)$ for every $i\in [d]$. Given a finite set of points $P\subset\mathbb{R}^d$, $(P,\preceq)$ is a $d$-dimensional poset. The Hasse diagram of this poset is a subgraph of the Delaunay graph of $P$ with respect to boxes.

Chen, Pach, Szegedy, and Tardos \cite{CPST} proved that there exists a set of $n$ points in the plane whose Delaunay graph with respect to rectangles has independence number $O(\frac{n(\log\log n)^2}{\log n})$. They used a similar argument to show that there are $2$-dimensional posets on $n$ vertices, whose Hasse diagram has roughly the same independence number. Therefore, one might wonder whether Theorem \ref{thm:Delaunay} can be also extended to Hasse diagrams of $d$-dimensional posets.

\begin{conjecture}
 There exists $c>0$ such that for every $d\geq 3$ and every sufficiently large $n$, there exists a $d$-dimensional poset, whose Hasse diagram has independence number at most $\frac{n}{(\log n)^{cd}}$. 
\end{conjecture}

 Suk and Tomon \cite{SukTom21} proved that an $n$-vertex Hasse diagram (with no restriction on its dimension) can have independence number $O(n^{3/4})$. We now repeat a problem from \cite{CPST}, which asks whether similar behavior can be achieved for bounded dimensional posets.

\begin{conjecture}
    For every $d\geq 2$ and $\varepsilon>0$, if $n$ is sufficiently large, then every $n$-vertex Hasse diagram of a poset of dimension $d$ has independence number at least $ n^{1-\varepsilon}$.
\end{conjecture}

\vspace{0.3cm}
\noindent	
{\bf Acknowledgements.} We would like the thank the anonymous referees for their useful insights and suggestions. Also, we would like to thank J\'anos Pach and G\'abor Tardos for the idea of Lemma \ref{lemma:independent} and for pointing out some errors, and James Davies and Zach Hunter for valuable comments.

\end{document}